\documentclass[twoside,a4paper,10pt]{amsart}

\usepackage{graphicx}                                     
                  
\usepackage{footmisc}
\numberwithin{equation}{section}

\usepackage{amsmath,amssymb,amsthm,amscd}

\usepackage{latexsym,amsfonts, verbatim}

\usepackage[all,cmtip,ps]{xy}

\usepackage[latin1]{inputenc}

\usepackage[english]{babel}

\newtheorem*{thm-plain}{Theorem}

\newtheorem{thm}{Theorem}[section]
\newtheorem{lem}[thm]{Lemma}
\newtheorem{cor}[thm]{Corollary}
\newtheorem{prp}[thm]{Proposition}

\theoremstyle{definition}
\newtheorem{dfn}[thm]{Definition}
\newtheorem{ntn}[thm]{Notation}
\newtheorem{rem}[thm]{Remark}

\newtheorem{exl}[thm]{Example}
\newtheorem*{exl-plain}{Example}
\newtheorem*{ack-plain}{Acknowledgements}

\theoremstyle{remark}

\newcommand{\A}{\vec A}
\newcommand{\N}{\mathbb{N}}
 
\newcommand{\mci}{\mathcal{I}} 

\newcounter{itri}
\renewcommand{\theitri}{(\roman{itri})}
\newenvironment{itr}{\begin{itemize}\setcounter{itri}{0}
\let\jtem\item
\def\item{\addtocounter{itri}{1}\jtem[\theitri]}}{\end{itemize}}
\newcounter{itai}
\renewcommand{\theitai}{(\arabic{itai})}

\begin{document}
\title{Dominant dimensions of two classes of finite dimensional algebras}
\author{ Muhammad Abrar}
\date{}
\begin{abstract}

The aim of this paper is to study the dominant dimension of two important classes
of finite dimensional algebras, namely, hereditary algebras and tree algebras. 
We derive an explicit formula for the dominant dimension of each class.

\end{abstract}
\maketitle

\section{Indroduction}
\let\thefootnote\relax\footnote{ This research  was fully supported by  Kohat University of Science and Technology (KUST),
Kohat, Pakistan, through the Human Resource Development (HRD) program of KUST.}

It is quite common to classify algebras by certain homological invariants. One such classification 
of finite dimensional algebras with respect to the length of an exact sequence of their projective-injective
bimodules was proposed by Nakayama \cite{nakayama58}. In \cite{tachikawak62} Tachikawa characterized $QF$-3
algebras by such length. Subsequently in 1964, Tachikawa \cite{tachikawak64} introduced the notion of dominant
dimension, where he studied the dominant dimension of $QF$-3 algebras as well. Later on, the classical theory of dominant dimension
has been developed by Mueller \cite{mueller68}, Tachikawa \cite{tachikawak70}, Morita \cite{morita70} 
and few others (e.g. \cite{suzuki71}). The dominant dimension also provides one of the two
conditions (the other one also is about a homological dimension: global dimension) in Auslander's \cite{auslander71} celebrated characterization of finite representation type,
that is of the representation category being finite.

In applied sense,  dominant dimension has been used not only to characterize the double centralizer property but also to 
classify certain algebras. In \cite{koenig2001} the dominant dimension has been used to prove several 
Schur-Weyl-dualities. Though the theory of dominant dimension is growing rapidly in applied context,
  see \cite{koenig2011,koenig2010}, the precise value of dominant dimension for many well-known classes 
of algebras is still unknown. Algebras of infinite dominant dimension also have been of interest of many ( e.g. \cite{auslander75})
in connection with Nakayama's conjecture, but the above perspective  also suggests to investigate the information itself about the
 dominant dimension (finite) of many important classes of finite dimensional algebras.

In this paper we study the dominant dimensions of  two well-known classes of algebras, namely hereditary algebras and tree algebras.
 We use quiver-theoretic techniques and give explicit combinatorial proofs of the results.

In Section \ref{sec: hereditary}, we consider hereditary algebras (quiver), and establish that a branching vertex plays a key role 
to characterize such class of algebras in terms of dominant dimensions. We conclude this section by
\begin{thm-plain}{\bf\ref{thm:hereditary}.}
 Let $A=KQ$ be a path algebra of a finite, connected and acyclic quiver $Q$. Then
 $$
       dom.dimA =
    \begin{cases}
       1 &\text{ if } Q=\A_n\\
       0 &\text{ if } Q\ne\A_{n}
       \end{cases}
    $$
 where $\A_n$ is linearly oriented.
\end{thm-plain}

In Section \ref{sec:An}, we pass through the quotients of $\A_n$, and  show that the quotients having free vertices have
dominant dimension not greater than two.
 We establish that
\begin{thm-plain}{\bf\ref{thm: lower-upper bound}.}
Let $A$ be a bound quiver algebra of $\A_n$. Then for a fixed $n\geq3$ $$1\leq dom.dimA\leq n-1.$$
\end{thm-plain}
It is also shown that these bounds are attained by quotients of $\A_n$, and that  every natural number occurs as dominant dimension. For the full set of fully overlapped
zero relations of the same length $m$, we  derive an explicit formula for $dom.dimA$. Indeed, in Proposition \ref{prp:formula A_n}, we prove that 
\emph{
\begin{itr} 
\item For $m=2$, $ dom.dimA=n-1$.
\item For $m\geq3$ and $n\in m\N+j$
\[
  dom.dimA=
    \begin{cases}
       \frac{2n-(m+2j)}{m} & j=1,2,3,\cdots ,m-2\\
       \frac{2n-2j}{m} & j=m-1\\
       \frac{2n-j}{m} & j=m.
    \end{cases}   \]         
\end{itr}}

In Section \ref{sec: general tree}, we study the dominant dimension of tree $(\neq\A_n)$ algebras. We define arms of a tree, and split trees into two classes, namely,
trees without arms and trees with arms. The following is an example of such two classes. 
\begin{exl-plain} The tree
 
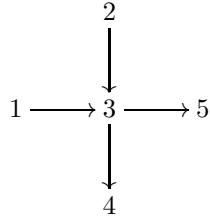
\[\begin{xy}
\xymatrix {
&2\ar[d]\\
1\ar[r]&3\ar[r]\ar[d]&5\\
&4
}
\end{xy}\]
has no arm, while the tree
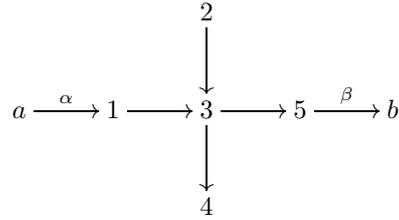
\[\begin{xy}
\xymatrix {
&&2\ar[d]\\
a\ar[r]^{\alpha}&1\ar[r]&3\ar[r]\ar[d]&5\ar[r]^{\beta}&b\\
&&4
}
\end{xy}\]
has two arms: $\alpha$ a left arm and $\beta $ a right arm.
\end{exl-plain}
Like hereditary algebras, it turns out that the dominant dimension of tree ($\neq\A_n$) algebras also can not exceed one. To deal with trees without arms,
we define ( see Definition \ref{dfn:*}) conditions $(\ast)$  as a set of relations satisfying:
\begin{itr}
 \item For each source $a$ and sink $c$, both $P(a)$ and $I(c)$ are uniserial.
\item For each $i\in Q_0'$, $socP(i)=\oplus S(c)$, where each $c$ is a sink.
\item  For each $i\in Q_0'$, $topI(i)=\oplus S(a)$, where each $a$ is a source.
\end{itr}
Consequently, we establish
\begin{thm-plain}{\bf\ref{thm:without arms}.}
   Let $R'$ be a set of zero relations on $Q'$. Then
    \[
       dom.dimB' =
    \begin{cases}
      1 &\text{ if } R'\text{ satisfies the conditions } (\ast)\\
       0 &\text{ otherwise.  }
       \end{cases}
    \]
\end{thm-plain}

In Subsection \ref{sec:with arms}, we pay attention to the trees with arms. Of course, trees with arms having, as set of relations, 
the conditions $(\ast)$ only act just like trees without arms, as Proposition \ref{prp:ddB=0} says. But in general, sets
of relations on trees with arms are bigger than and might be containing the  conditions $(\ast)$. Hence this leads to the 
conditions $(\ast\ast)$, an extension of the conditions $(\ast)$. We define such conditions as (see Definition \ref{dfn:**}):

Let $R$ and $R'$ be sets of zero relations on $Q$ and $Q'$, respectively, such that $R'\subseteq R$. Then $R$ is said to
satisfy the conditions $(\ast\ast)$  if
 \begin{itr}
\item $R'$ satisfies the conditions $(\ast)$.
\item $\forall $ $i$ in left arm, $socP(i)=S(i')$ for some successor $i'\notin Q_0''$ of $i$.
\item  $\forall $ $j$ in right arm, $topI(j)=S(j')$ for some predecessor $j'\notin Q_0''$ of $j$.
\end{itr}
Consequently, we have
\begin{thm-plain}{\bf\ref{thm:with arms}.}
 Let $R$ and $R'$ be sets of zero relations on $Q$ and $Q'$ respectively, 
such that $R'\subseteq R$ and $R\cap S'=R'$. Then
  \[
       dom.dimB =
    \begin{cases}
       1 &\text{ if } R \text{ satisfies the conditions } (\ast\ast)\\
	0 &\text{ otherwise.}
       \end{cases}
    \]
  
\end{thm-plain}

This paper is a part of a comprehensive project on dominant dimensions where finite dimensional algebras
are to be characterized explicitly by precise values or by a range of values of their dominant dimensions.

\section{Preliminaries}
Here we recall some basic notions from quiver theory and make some useful conventions. We also give few elementary results.

 Throughout, $K$ is assumed to be a field, and  $Q=(Q_0,Q_1,s,t)$  a finite, connected and acyclic quiver,
and   $\A_n$ a linearly oriented quiver having  $Q_0=\{1,2,3\cdots,n\}$ as the set of vertices,
 where $n\in\N$.
 We call $Q$ a \emph{tree} if there is a unique path between any two vertices in $Q_0$.
Let $x$ be a path in $Q$. We denote by $Q_0^x$ and $Q_1^x$ respectively the set of all vertices in $x$ and
the set of all arrows in $x$.  We say the path $x$ contains a vertex $a$ if $a\in Q_0^x$.
If there exists in $Q$ a path from $a$ to $b$, then $a$ is said to be a \emph{predecessor} of
$b$, and $b$ is said to be a \emph{successor} of $a$. In particular, if there exists
an arrow $a\rightarrow b$, then $a$, written $b^{-}$, is said to be an \emph{immediate predecessor} of
$b$, and $b$, written $a^{+}$, is said to be an \emph{immediate successor} of $a$.
We define a \emph{ relation } in $Q$ with coefficients in $K$ as a $K$-linear combination of paths of length
at least two having the same source and target. A \emph{ zero }(or monomial) relation in $Q$ is a relation comprising 
only one term of $K$-linear combination,  see \cite{assem} for details. Any zero relation, by definition, is minimal.
By \emph{length} of a relation we mean the number of arrows in the relation. The source and the target
of a zero relation are defined as the source (target) of the first (last) arrow in the zero relation.
A path $x$ of length at least one in  $Q$ is said to be \emph{maximal} if it is not a subpath of any other path in $KQ$ or $KQ/\mci$.

All the projective $P(j)$  and the injective  $I(j)$ modules under consideration are the indecomposable left $A$-modules
corresponding to some vertex $j\in Q_0$, where $A$ is either a path algebra or a bound quiver algebra.
 Any projective (injective) module which is injective (projective), up to isomorphism, will
be called \emph{ projective-injective}. Any zero module, by definition,  is projective-injective.
Throughout, an injective envelope of a module $M$ is denoted by $EM$.

The following definition has a fundamental role when dealing with hereditary and tree algebras.
\begin{dfn} A vertex $a$  in  $Q$  is said to be a \emph{branching vertex}
if there exist distinct arrows  $\alpha,\beta\in Q_1$ such that $s(\alpha)=a=s(\beta)$  or $t(\alpha)=a=t(\beta)$.

By definition, $\A_n$ is a branching-free tree.
\end{dfn}
\begin{dfn} 
\label{dfn:free vertex}
A vertex $a$ in $Q$  is said to be a \emph{ free vertex} if it is neither the source nor
the target of any zero relation.
\end{dfn}
\begin{dfn}
Let $M$ be an $A$-module. $M$ is said to have the \emph{ dominant dimension } at least $n\in \N$,
 written $dom.dimM\geq n$, if there exists a minimal injective resolution 
  \[0\rightarrow M\rightarrow I_1\rightarrow I_2\rightarrow\cdots \rightarrow I_n\rightarrow\cdots\]
of $M$ such that all the modules $I_j$ with $1\leq j\leq n$ are projective-injective.
\end{dfn}
If the injective envelope $I_1$ of $M$ is not projective, we set $dom.dimM=0$. In case $dom.dimM\geq n$ and $dom.dimM\ngeq n+1$,
we say $dom.dimA= n$. If no such $n$ exists, we write $dom.dimM=\infty$.
The dominant dimension of an algebra $A$ is defined as the dominant dimension of the left regular 
module $_{A}A$, that is,   $dom.dimA=dom.dim_{A}A$.
 A self-injective algebra has infinite
dominant dimension, since all of its projective modules are injective.
 An obvious consequence of the definition is that 
 $$dom.dim(M\oplus N)=min (dom.dimM,dom.dimN)$$ where  $M$ and $N$
are finite dimensional $A$-modules. Because the dominant dimension of a projective-injective module is infinite,
 this consequence implies to forget the trivial part ( where every projective is injective )
of the minimal injective resolution of $_{A}A=\oplus M_i$.

\begin{lem}
 \label{lem:minimal}
Let 
$I^{\bullet}: 0\rightarrow M\rightarrow I_1\rightarrow I_2\rightarrow\cdots \rightarrow I_n\rightarrow0$
be an injective resolution of a non-injective module $M$ such that $I_j$ with $1\leq j\leq n-1$ is projective.
If $I^{\bullet}$ is minimal, then $I_n$ is not projective.
\end{lem}
\begin{proof}
 Let us assume, on the contrary, that $I_n$ is  projective. Then the epimorphism
$ I_{n-1}\rightarrow I_n\rightarrow0$ splits. This implies that $I_n$ is a direct summand of $I_{n-1}$.
Consequently,  $I^{\bullet}$ is not minimal, but $I^{\bullet}$ was minimal. Hence $I_n$ is not projective.

\end{proof}

In general, an upper bound of $dom.dimA$ for many algebras is not known yet. But, in particular, directed algebras are bounded above by 
the number of their projective-injective modules, as shown below.
\begin{thm}
 \label{thm:lub}
Let $A$ be a bound quiver algebra of a finite, connected and acyclic quiver $Q\neq\A_1$. Then $dom.dimA\leq d\leq n-1$, where 
$|Q_0|=n\in\N$ and $d$ is the number of projective-injective $A$-modules. 
\end{thm}
\begin{proof}
 Because $A$ is not self-injective, there are 
at most $n-1$  projective-injective $A$-modules, and therefore $d\leq n-1$.
Since $Q$ has no oriented cycles, $A$  is a directed algebra. Let  
 \[I^{\bullet}: 0\rightarrow _AA\rightarrow I_1\rightarrow I_2\rightarrow\cdots \rightarrow I_k\rightarrow I_{k+1}\rightarrow I_{k+2}\rightarrow\]
be a minimal injective resolution
of $A$ such that $I_j=\oplus I(a_{ij})$ is non-zero projective for $1\leq j\leq k$, where $I(a_{ij}\in Q_0)$ is indecomposable
 projective-injective $A$-module, $1\leq i\leq n_j$ and $n_j$ is the number of direct summands of $I_j$.
Now for $j=1,2,\cdots,k-1$, each  matrix $\phi_j:I_j\longrightarrow I_{j+1}$ has as entries the morphisms
$\phi_{ij}:I(a_{ij})\longrightarrow I(b_{i'j+1})$ where  $a_{ij},b_{i'j+1}\in Q_0$. 
Since $A$ is a directed algebra, so $\phi_{ij}$ is non-zero, non-invertible, and for each  $j=1,2,\cdots,k-1$,
 $a_{ij}\lneq b_{i'j+1}$ . This implies that the injective resolution $I^{\bullet}$ is finite, as $|Q_0|=n$ is fixed.
Hence there exists some $j=k+2$ (say) such that $I_j=0$ and $k+2\leq n+1$.
Since $I^{\bullet}$ is minimal, it follows from Lemma \ref{lem:minimal} that $I_{k+1}$ is not projective. 
Consequently,
\[dom.dimA=k=\sum_{1\leq j\leq k}j\leq\sum_{1\leq j\leq k}n_j=d\leq n-1\]
or $dom.dimA\leq d\leq n-1$. This completes the proof.

\end{proof}

The following Lemma is used frequently to settle many results.
\begin{lem}
\label{lem:new maximal path}
Let $B$ be a bound quiver algebra of a tree $Q$ and $a,b\in Q_0$  two  distinct vertices.
Then $P(a)\cong I(b)$ if and only if there exists a path $x$ from $a$ to $b$ such that $x$ is maximal, and both $P(a)$ and $I(b)$
are uniserial. In particular, if $Q=\A_n$,  $P(1)$ and $I(n)$ are projective-injective.
\end{lem}
\begin{proof}
 Suppose $P(a)\cong I(b)$. Then obviously $topP(a)=S(a)=topI(b)$ and $socI(b)=S(b)=socP(a)$. This shows existence of the path 
$x$ from $a$ to $b$. Now $x$ is maximal, because otherwise either $socP(a)\neq S(b)=socI(b)$ or $topI(b)\neq S(a)=topP(a)$,
which is contrary to the supposition.
Since $Q$ is a tree, and both $P(a)$ and $I(b)$
have simple socle and simple top, therefore both $P(a)$ and $I(b)$ are uniserial.

Conversely, assume that there exists a maximal path $x$ from $a$ to $b$, and  both $P(a)$ and $I(b)$ are uniserial. Because  $x$ is a maximal path from
$a$ to $b$, and  $P(a)$ is uniserial, therefore $socP(a)=S(b)$. This shows that $EP(a)=I(b)$. 
Thus  $P(a)\hookrightarrow I(b)$. Now since $I(b)$ is also uniserial, the maximality of the path $x$ gives
$topI(b)=S(a)$. This implies that the projective cover of $I(b)$ is $P(a)$, and hence $P(a)\twoheadrightarrow I(b)$.
Consequently,  $P(a)\cong I(b)$.

Obviously, all the projectives and injectives are uniserial if $Q=\A_n$.
 Let  $socP(1)=S(j\neq n)$ and $topI(n)=S(i\neq1)$ where
$i\neq j\in Q_0$. Since $1\in Q_0$ being a source has no predecessors, the path from 1 to $j$ is maximal.
Hence $P(1)\cong I(j)$. Similarly, the path from $i$ to $n$ is maximal because $n$ being a sink has no successors.
Thus $I(n)\cong P(i)$.

\end{proof}

\section{Hereditary algebras}
\label{sec: hereditary}
Throughout the section, except in Proposition \ref{prp:An}, it is assumed that 
 $Q\neq \A_n$ and  $A=KQ$ is a path algebra of $Q$, where $\A_n$ is a linearly oriented quiver with $n$ vertices.

\begin{prp}
\label{prp:An}
The path algebra of  $\A_n$ has dominant dimension equal to one.
         
\end{prp}

\begin{proof}
Let $A$ be the path algebra of $\A_n$.
We first show that $I(n)$ is the only injective which, up to isomorphism, is projective. Since every
projective  $P(i)$ has the simple socle $ S(n)$, therefore $EP(i)=I(n)$. As $P(1)$ and  $I(n)$ have the same dimension, $P(1)\cong I(n)$.
Next we show that for all $j=1,2,...,n-1$, $I(j)$ is not projective. Suppose, on the contrary, that $I(j)$ is projective. That is, $I(j)\cong P(k)$
for a $k$ such that $dimI(j)=dimP(k)$. Now $P(k)\cong I(j)$ implies  $EP(k)=I(j)$ where $j\neq n$, a contradiction to the fact that
$EP(k)=I(n)$.
Thus $I(j)$ is not projective for $j=1,2,...,n-1$.
Now the minimal injective resolution  of $A$ becomes
                  \[0\rightarrow A\rightarrow I(n)^n\rightarrow I(1)\oplus I(2)\oplus \cdots \oplus I(n-1)\rightarrow 0\]
where $I(n)^n$ is the direct sum of $n$ copies of $I(n)$. Hence $dom.dimA=1$.
\end{proof}

\begin{lem}
\label{lem:longest path}
        Every longest path in $Q$ contains at least one branching vertex.
\end{lem}

\begin{proof}
    Let $x$ be a longest path in $Q$ and  $ Q_0^x=\{a_1,\ldots,a_n\}$ such that $ s(x)=a_1$ and
$t(x)=a_n $. Since $x$ is longest, so there does not exist a path, say $w$ such that $s(w)=a_n$ or $ t(w)=a_1$. 
Now there may or may not exist a vertex $b\in Q_0\smallsetminus Q_0^x$. First assume that there exists
 a $b\in Q_0\smallsetminus Q_0^x$.
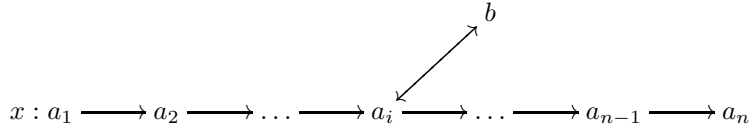
\[\begin{xy}
\xymatrix {
&&&&b\ar[ld]\\
x:a_1\ar[r]&a_2\ar[r]&\dots \ar[r]&a_i\ar[ru]\ar[r]&\dots\ar[r]&a_{n-1} \ar[r]&a_n
}
\end{xy}\]
Since $Q$ is connected, there exist: an unoriented path,
say $y$  between $b$ and some $a_i\in Q_0^x $, and an arrow $\alpha \in Q_1^y$ such that 
  \begin{equation}
\label{eq: branching1}
    \begin{array}{cl}
s(\alpha)=a_i=s(x)&\text{ if } i=1\\
t(\alpha)=a_i=t(x)&\text{ if } i=n
 \end{array}
  \end{equation}
and if $1<i<n$ then 
\begin{equation}
\label{eq: branching2}
  \begin{array}{cl}
s(\alpha)=a_i=s(\beta)&\text{ for some } \beta \in Q_1^x\\
\text{ or}\\
t(\alpha)=a_i=t(\gamma)&\text{ for some } \gamma \in Q_1^x
 \end{array}
  \end{equation}
Now by definition, $a_i$ is a branching vertex. 

Next suppose there does not exist any $b\in Q_0\smallsetminus Q_0^x$. Then $Q_0^x=Q_0$.
\[ x:a_1\longrightarrow a_2\longrightarrow\cdots \longrightarrow a_i\longrightarrow\cdots \longrightarrow a_{n-1}\longrightarrow a_n\]
 Since $Q\neq \A_n$, we
 have multiple arrows in Q. Hence there must exist at least one arrow $\alpha \in Q_1\smallsetminus Q_1^x$ and $a_i\in Q_0^x=Q_0$ such that
the equations \eqref{eq: branching1} and \eqref{eq: branching2} hold.  Hence  $a_i$ is a branching vertex. 
 For such an arrow $\alpha:a_i\longrightarrow a_j$, necessarily $j>i$ because otherwise
 we would have a cycle, contradicting the assumption that $Q$ is acyclic. Also note that if there does not exist an
  $\alpha \in Q_1\smallsetminus Q_1^x$, then $Q_1^x=Q_1$. This, together with $Q_0^x=Q_0$, implies that  $Q=\A_n$, again a contradiction.
\end{proof}

\begin{lem}
\label{lem:2 simple summands}
 Let $x$ from $a$ to $b$ be a longest path in $Q$. Then $ socP(a)$ or   $ topI(b)$ has at
        least two simple summands.
\end{lem}
\begin{proof}
        Since  $x$ is a longest path in $Q$, it follows from Lemma \ref{lem:longest path} that $x$ contains at least one
        branching vertex. Thus we can assume that $c$ or $d$ in $ Q_0^x$ are the branching vertices. 
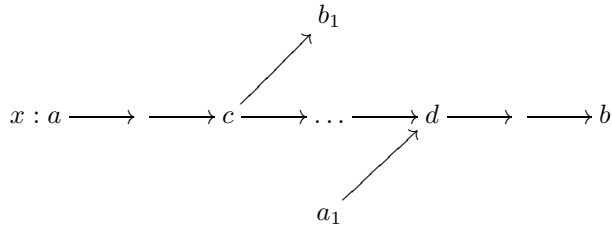
\[\begin{xy}
\xymatrix {
&&&b_1\\
x:a\ar[r]&\ar[r]&c\ar[ru]\ar[r]&\dots\ar[r]&d\ar[r]&\ar[r]&b\\
&&&a_1\ar[ru]
}
\end{xy}\]
 Assume that $x_1$ and $x_2$ be two paths with $ s(x_1)= a=s(x)$,  $t(x_1)= b_1$  and
$s(x_2)= a_1$, $t(x_2)= b=t(x)$ such that $ \ell(x_i)\leq \ell(x)$ for $i=1,2$, where $b_1$ is the sink and $a_1$ is the source,
and $\ell(x)$ is the length of the path $x$. Hence $x_1$, $x_2$ are not the subpaths.

        Now since $x_1$ is not a subpath and $b_1$ is a sink, so $x_1$ is maximal and hence annihilated by the radical $rad(A)$ of the path
        algebra $A=KQ$. \emph{i.e.} $rad(A).x_1=0$. This implies that $ S(b_1)$ is a summand of $ soc P(a)$. Also $ S(b)$ is a summand of $ soc P(a)$.
        Note that if $b_1=b$ then we have two copies of the simple $A$-module $S(b)$ in the socle. Hence $ soc P(a)$ has at least two simple
        summands.

        Similarly $x_2$ is not a subpath and hence is maximal, so that $<x_2^*>\subset top(I(b))$. Hence the simple $A$-modules $S(a)$ and
        $S(a_1)$ are summands of $topI(b)$. If $a_1=a$ then we have two copies of the simple $A$-module $S(a)$ in the $ topI(b)$. Thus
        $ topI(b)$ has at least two simple summands. This proves the Lemma.
\end{proof}
\begin{lem}
\label{lem:non-projective env}
        Let $x$ from $a$ to $b$ be a longest path in $Q$. Then $I(b)$ is a non-projective summand of $EP(a)$.
\end{lem}
\begin{proof}
        Since $x$ is a longest path in $Q$, so the proof of Lemma \ref{lem:2 simple summands} gives 
\[  socP(a)=S(b)\oplus \cdots \oplus\quad\text{ and } \qquad  topI(b)=S(a)\oplus \cdots \oplus\]
Thus
 \[ EP(a) = E(socP(a))= E(S(b)\oplus \cdots \oplus)= ES(b)\oplus \cdots \oplus= I(b)\oplus \cdots \oplus     \]
Hence $I(b)$ is a direct summand of $EP(a)$.

        Next suppose, on the contrary, that $I(b)$ is projective. Then $I(b)\cong P(a)$ gives $EP(a)=I(b)$. 
This implies that $socP(a)$ and the $topI(b)$ are simple, a contradiction to the fact that
	$ soc P(a)$ or  $ topI(b)$ has at least two simple summands.
        Hence $I(b)$ is not projective.
\end{proof}
\begin{thm}
        The path algebra $A$ of $Q\neq \A_n$  has dominant dimension equal to zero.
\end{thm}
\begin{proof}
It is enough to prove that $I_1=\oplus EP(i)$ in the minimal injective resolution
of $A$ contains a non-projective summand. 
Since $Q$ is finite, connected and acyclic, it contains a longest path, say, from $a$ to $b$.
 Then by Lemma \ref{lem:non-projective env}, $I(b)$ is a non-projective
 summand of $EP(a)$ and hence of $I_1$. Hence $I_1$ is not projective showing that $dom.dim A=0$.
\end{proof}
We summarize this chapter as
\begin{thm} Let $A=KQ$ be a path algebra of a finite, connected and acyclic quiver $Q$. Then
\label{thm:hereditary}
       $$
       dom.dimA =
    \begin{cases}
       1 &\text{ if } Q=\A_n\\
       0 &\text{ if } Q\ne\A_{n}.
       \end{cases}
    $$
\end{thm}

In the following two sections we concentrate on the bound quiver algebras of finite trees.
Before we proceed further, we observe from the above section that a branching vertex has a central role in computing the
dominant dimension of a path algebra of a finite, connected and acyclic quiver. Since every tree, except $\A_n$, has at least one
branching vertex, it motivates us to consider first the bound quiver algebras of the branching free tree $\A_n$.

\section{Bound quiver algebras of $\mathbb{ A}_n$}
\label{sec:An}

We consider the quotient algebras of $\A_n$ for $n\geq3$.
Throughout this section, we assume that $Q=\A_n=Q_n$ and that $A=KQ/\mci$ is a bound quiver algebra of $Q$, where $\mci$ is an admissible of $KQ$ generated by 
a certain set of zero relations. We go through different sets of zero relations to investigate how 
$dom.dimA$ depends on the choice of zero relations. We find  lower and upper bound of $dom.dimA$ and show
by examples that these bounds are optimal.

  For convenience, we denote by  $dd(P(i),Q_j)$   the dominant dimension of the  projective module 
$P(i)$  when $Q=Q_j$ where $3\leq j<n$. If  $Q=Q_n$, we write $dd(P(i),Q_n)=ddP(i)$.
Given a set $R$ of zero relations on $Q$, we denote by $R_0^{s}$ and $R_0^{t}$, respectively, the set of sources and the set of targets of the relations
in $R$. Obviously, $R_0^{s},R_0^{t}\subsetneq Q_0$,   $1,2\notin R_0^{t}$ and $n,n-1\notin R_0^{s}$ for every set $R$ of zero relations on $Q$.

\begin{lem}
\label{lem:inj-proj env}
Let $R$ be a set of zero relations on $Q$. Then $EP(a)$ is projective, for every $a\in Q_0$.
\end{lem}
\begin{proof}
Let $a\in Q_0$ be an arbitrary vertex. Then $a$ may or may not be the source of a maximal path. 
If $a$ is the source of a maximal path, say $x$  with target $t(x)$, then by Lemma \ref{lem:new maximal path}
 $P(a)\cong I(t(x))$. Hence the injective envelope $EP(a)=I(t(x))$ of $P(a)$ is projective.

Now assume that $a$ is not the source of any maximal path. Let $y$ be the longest path with source $a$.
Then $EP(a)=E(socP(a))=E(S(t(y)))=I(t(y))$. The lemma follows if we show that $t(y)$ is the target of some maximal path. Since the
 path $y$ is not maximal, there exists the smallest predecessor $c$ of $a$ such that the path $z$ from $c$ to $t(y)$ is maximal. 
Thus $t(z)=t(y)$ and hence $I(t(y))=I(t(z))$ is projective.

 \end{proof}

The following Proposition gives  lower bound of $dom.dimA$.
\begin{prp}
\label{prp:lower bound}
Let $A$ be a bound quiver algebra of $\A_n$. Then\\ $dom.dim A\geq1$.
\end{prp}
\begin{proof}
	It is enough to prove that $I_1$ 
is projective. From Lemma \ref{lem:inj-proj env}, it follows that $EP(a)$ is projective for every $a\in Q_0$.
 Thus $I_1=\bigoplus\limits_{a\in Q_0} EP(a)$ is projective and hence $dom.dim A\geq1$.
\end{proof}

\begin{thm}
 \label{thm: lower-upper bound}
Let $A$ be a bound quiver algebra of $\A_n$. Then for a fixed $n\geq3$ $$1\leq dom.dimA\leq n-1.$$
\end{thm}
\begin{proof}
 It follows immediately from Proposition \ref{prp:lower bound} that $dom.dimA\geq1$.
Since $A$ is a directed algebra, Theorem \ref{thm:lub} implies $ dom.dimA\leq n-1$.
This proves the Theorem.
\end{proof}

\begin{lem}
 \label{lem: source sink free}
Let $R$ be a set of relations such that the source (sink) of $\A_n$ is free. Then $dom.dimA=1$.
\end{lem}
\begin{proof}
We have $dom.dimA\geq1$ from Proposition \ref{prp:lower bound}. We need to find a projective $P$ such that $ddP\ngtr1$.

First we assume that the source of $\A_n$ is free, that is, 1 is not the source of any zero relation in $R$.
Then the path from $1$ to $ t^{-}$ is maximal, where $t\in R_0^{t}$ is smallest. 
Hence $P(1)\cong I(t^{-})$ and $socP(2)=S(t^{-})$.
 Now $ddP(2)=1$, as obvious from the following resolution
\[0\rightarrow P(2)\rightarrow I(t^{-})\cong P(1)\rightarrow  I(1)\rightarrow 0\]
where $I(1)$ is not projective. Hence $dom.dimA=1$.

Next, suppose that the sink $n$ of $\A_n$ is free. Then $I(n)\cong P(s^{+})$ and $topI(n-1)=S(s^{+})$ where $s\in R_0^{s}$ is largest.
Because $I(n-1)$ is not projective, the resolution
\[0\rightarrow P(n)\rightarrow I(n)\rightarrow  I(n-1)\rightarrow0 \]
shows that $ddP(n)=1$, and ultimately $dom.dimA=1$.
\end{proof}

In view of Lemma \ref{lem: source sink free}, from now on  we assume that 
 $1\in R_0^{s}$ and  $n\in R_0^{t}$ for every set $R$ of zero relations on $Q$.

\begin{prp} Let $R$ be a set of zero relations on $Q$. 
 \label{prp: free}
\begin{itr}
\item If  $a\notin R_0^{s}$ and  $a^{+}\notin R_0^{t}$ for some $a\in Q_0$, then $dom.dimA=1$.
 \item  If $R$ is such that $Q_0$ contains a free vertex, then $dom.dimA\leq2$.
\end{itr}
\end{prp}
\begin{proof} 
(i) Assume that  $a\notin R_0^{s}$ and  $a^{+}\notin R_0^{t}$ for some $a\in Q_0$.
 Then $P(a^{+})$ is not injective and $I(a)$ is not projective.
 Now it is easy to see that  $P(s^{+})\cong I(t^{-})$, where 
$s\in R_0^{s}$  is largest but $s\leq a^{-}$ and  $t\in R_0^{t}$  is the target of zero relation with 
smallest source $s'\geq a^{+}$ if exists, otherwise $P(s^{+})\cong I(n)$.
Since  $socP(a)=socP(a^{+})=S(t^{-})$ or $S(n)$, we have the injective resolution
\[0\rightarrow P(a^{+})\rightarrow  P(s^{+})\cong I(t^{-})\text{ or } I(n)\rightarrow  I(a)\rightarrow \]
with $I(a)$ not projective. Hence  $dom.dimA=1$.

(ii) Let $a$ be a free vertex for $R$. All we need is to find a projective $P$ such that $ddP\ngtr2$. If  $a^{-}\notin R_0^{s}$ or
$a^{+}\notin R_0^{t}$, then it follows from (i) that $dom.dimA=1$, since $a\notin R_0^{s}$ and $a\notin R_0^{t}$.
Suppose $a^{-}\in R_0^{s}$  and   $a^{+}\in R_0^{t}$. Then the path from $a$ to $t^{-}$ is maximal, where  $t\in R_0^{t}$  
is the target of a zero relation with smallest source $s'\gneq a$. Consequently,  $P(a)\cong I(t^{-})$ and  $socP(a^{+})=S(t^{-})$.
Similarly path from $s^{+}$ to $a$ is maximal, where $s\in R_0^{s}$ is the source of a zero relation 
with largest target $t'\lneq a$. Hence it follows that $ P(s^{+})\cong I(a)$ and $topI(a^{-})=S(s^{+})$. Now $P(a^{+}) $ has the resolution 
\[0\rightarrow P(a^{+})\rightarrow I(t^{-})\cong P(a)\rightarrow  I(a)\cong P(s^{+})\rightarrow I(a^{-})\rightarrow 0\]
where $ I(a^{-})$ is not projective, and therefore $ddP(a^{+})=2$. Hence $dom.dimA$ can not exceed two, or $dom.dimA\leq2$.
\end{proof}
In the following we list some of those sets which satisfy the conditions of the Proposition \ref{prp: free}.
\begin{rem}
We observe that there always exists a free vertex for the following sets $R$ of relations on $\A_n$, where $n$ is fixed.
\begin{enumerate}
 \item Every two relations in $R$ are disjoint.
\item For all $ a_1, a_2\in R_0^{s}, \quad  a_1<a_2$ implies $ m_1<m_2$ ($ m_1>m_2$) where 
$m_i$ is the length of relation starting at $a_i$.
\end{enumerate}
If for all $ a_1, a_2\in R_0^{s}, \quad  a_1<a_2$ implies $ m_1<m_2$ or $ m_1>m_2$, 
 then either there exists a free vertex, or $R$ satisfies (i) of Proposition \ref{prp: free}.
 Hence in each case dominant dimension can not be greater than two.
If $R$ is such that $a\in\{2,3,\cdots,\ m-1\}\cup \{n-m'+2,\cdots,n-1\}$ is free,
where $m$ ($m'$) is the length of zero relation starting (ending) at 1 ($n$). Then, in fact, 
$R$ again satisfies (i) of Proposition \ref{prp: free}.

Thus, in order to generate larger dominant dimensions, it is now essential to consider the sets of
fully overlapped zero relations on $Q$.
\end{rem}

\begin{prp}
\label{prp:formula A_n}
Let $R$ be the full set of fully overlapped zero relations of the same fixed length $m\geq2$.
\begin{itr}
 \item For $m=2$, $ dom.dimA=n-1$.
\item For $m\geq3$
\[
  dom.dimA=
    \begin{cases}
       \frac{2n-(m+2j)}{m} & j=1,2,3,\cdots ,m-2\\
       \frac{2n-2j}{m} & j=m-1\\
       \frac{2n-j}{m} & j=m
    \end{cases}   \]         
\end{itr}
where $n\in m\N+j$.
\end{prp}
\begin{proof} (i) If $m=2$,
then  $P(i)\cong\ I(i+1)$ for all $ i=1,2,\cdots,n-1$, and thus
$P(n)$ is the only projective which is not injective.
The minimal injective resolution of  $P(n)$ becomes
{\footnotesize
 \[0\rightarrow
 \begin{matrix}
        n
 \end{matrix} \rightarrow 
\begin{matrix}
  n-1\\
   n
 \end{matrix} \rightarrow 
 \begin{matrix}
n-2 \\
\boldsymbol{ n-1}                                                              
\end{matrix}\rightarrow
 \begin{matrix}
n-3\\
\boldsymbol{ n-2}                         
\end{matrix}\rightarrow \cdots \rightarrow
\begin{matrix}
2\\
\boldsymbol{3}
 \end{matrix}\rightarrow
\begin{matrix}
 1\\
\mathbf{2}
\end{matrix}\rightarrow
\begin{matrix}
\mathbf{1}
\end{matrix}\rightarrow 0\]
}where $I(1)$ is not projective, and hence $ dom.dimA=ddP(n)=n-1$.

(ii) The projectives  $P(i=n-m+2,n-m+3,\cdots,n-1,n)$ are not injective.
We first find $ddP(n-m+2)$. We note that the injectives $I(i=1,2,\cdots,m-1)$ are not projective.

Consider the resolution 
{\scriptsize
    \[
     0\rightarrow
 \begin{matrix}
        n-m+2 \\
        n-m+3 \\
        \vdots\\
        n
 \end{matrix} \rightarrow \begin{matrix}
                                    n-m+1 \\
                                    n-m+2 \\
                                    \vdots\\
                                    n
                             \end{matrix} \rightarrow  \begin{matrix}
                                                                n-2m+2 \\
                                                                n-2m+3\\
                                                                \vdots\\
                                                               \boldsymbol{n-m+1}\\
							       \end{matrix}\rightarrow 
\begin{matrix}
  n-2m+1\\
\boldsymbol{n-2m+2}\\
\vdots\\
\boldsymbol{n-m}\\
\end{matrix}\rightarrow
\begin{matrix}
 n-3m+2\\
 n-3m+3\\
\vdots\\
\boldsymbol{n-2m+1}
\end{matrix}\rightarrow 
\begin{matrix}
  n-3m+1\\
\boldsymbol{n-3m+2}\\
\vdots\\
\boldsymbol{n-2m}\\
\end{matrix}\rightarrow
\]}where the second cokernel
 \[\begin{matrix}  
 n-2m+2 \\
 n-2m+3\\
\vdots\\
 n-m
\end{matrix} \]
is the indecomposable projective $KQ_{n-m}/\mci_m$-module  $P'(n-2m+2)$, and $ Q_{n-m}$ with $ (Q_{n-m})_{0}=\{1,2,\cdots,n-m\}$ is a subquiver of
$Q_n$. Because exactly two injectives  are projective to obtain the subquiver $Q_{n-m}$, and
since all the projective-injective $KQ_{n-m}/\mci_m$-modules are the projective-injective $A$-modules, so we have
 \[ ddP(n-m+2)=2+dd(P'(n-2m+2),Q_{n-m})\]
where  $\mci_m$ is an admissible ideal of $KQ_{n-m}$ generated by the full set of zero relations
of the same length $m$. By the similar arguments, we obtain
\[ dd(P'(n-2m+2),Q_{n-m})=2+dd(P''(n-3m+2),Q_{n-2m})\]
Therefore $dd(P(n-m+2),Q_n) $  becomes
\[ \begin{array}{ccl}
        ddP(n-m+2) & = &2+dd(P'(n-2m+2),Q_{n-m})\\
                        & = & 2+2+dd(P''(n-3m+2),Q_{n-2m}).
        \end{array}
        \]
Hence  proceeding in this way, we get
\begin{equation}
\label{eq r}
         \begin{array}{ccl}
        ddP(n-m+2) & = &\overbrace{2+2+\cdots+2}^{x \text{ times }}+dd(\dot{P}(r-m+2),Q_r)\\
                        & = & 2x+dd(\dot{P}(r-m+2),Q_r)
        \end{array}
\end{equation}
 where $n\in m\N+j$, $r=m+j$ with $j=1,2,\cdots,m-2,m-1,m$, and
$x$ times $m$ is subtracted from $n$ to obtain such $r$. 
Thus \[n-mx=r \Rightarrow x=\frac{n-r}{m}\quad\text{ and } \quad2x=\frac{2n-2r}{m}\]
Substituting the values of $x$ and $r$ in \eqref{eq r} and get 
 \[ ddP(n-m+2) =     \frac{2n-2(m+j)}{m}+dd(\dot{P}(j+2),Q_{m+j})\]
where  $j=1,2,\cdots,m-1,m$.
Now the same process gives
 \[
        \begin{array}{ccl}
         ddP(n-m+3) & = & \frac{2n-2(m+j)}{m}+dd(\dot{P}(j+3),Q_{m+j}) \\
\vdots&&\vdots\\
         ddP(n-m+m-1) & = & \frac{2n-2(m+j)}{m}+dd(\dot{P}(j+m-1),Q_{m+j})\\
        ddP(n) & = &  \frac{2n-2(m+j)}{m}+dd(\dot{P}(j+m),Q_{m+j})
        \end{array}
    \]

Next we prove that
\begin{equation}
\label{ineq1}
     ddP(n-m+2)\leq ddP(k)
\end{equation}
for all  $ k=n-m+3,n-m+4,\cdots,n-1, n$.
To prove \eqref{ineq1}  we have to show  for all  $ j=1,2,\cdots,m-2,m-1,m$ that
\begin{equation}
\label{ineq2}
   dd(\dot{P}(j+2),Q_r)\leq dd(\dot{P}(j+i),Q_r) 
\end{equation}
where $ i=3,4,\cdots,m$. Let us consider the resolution
 {\tiny\[
     0\rightarrow
 \begin{matrix}
        j+2 & j+3&\cdots &j+m\\
        j+3 &\vdots \\
         \vdots &j+m\\
        j+m
\end{matrix} \rightarrow 
\begin{matrix}
j+1 &\cdots &j+1 \\
j+2 &&j+2\\
\vdots & &\vdots\\
j+m&&j+m
\end{matrix} \rightarrow
\begin{matrix}
j-m+2 &j-m+3 &\cdots &j\\
j-m+3 &j-m+4 & &\boldsymbol{j+1}\\
\vdots & \vdots &&\vdots \\
j& \boldsymbol{j+1}&& \boldsymbol{j+m-2} \\
\boldsymbol{j+1}&\boldsymbol{j+2}&&\boldsymbol{j+m-1}
\end{matrix}\rightarrow\]

\[ \rightarrow
\begin{matrix}
j-m+1&j-m+1 &\cdots &j-m+1\\
\boldsymbol{j-m+2}&j-m+2 &&j-m+2\\
\boldsymbol{j-m+3}&\boldsymbol{j-m+3}&&j-m+3\\
\vdots&\vdots&& \vdots  \\
\boldsymbol{j}&\boldsymbol{j}&&\boldsymbol{j}
\end{matrix} \rightarrow
\begin{matrix}
\boldsymbol{j-m+1} &\boldsymbol{j-m+1}&\cdots &\boldsymbol{j-m+1}\\
&\boldsymbol{j-m+2}&&\boldsymbol{j-m+2}\\
&&&\boldsymbol{j-m+3} \\
&&& \vdots  \\
&&&\boldsymbol{j-1}
\end{matrix} \rightarrow 0 \]
}Since $I(j+1)$ is not projective for each $j=1,2,\cdots,m-2$, 
it follows from the above resolution that 
       \[ dd(\dot{P}(j'),Q_r)=1<2= dd(\dot{P}(i'),Q_r)\]
where  $ j'=j+2,j+3,\cdots,m$ and  $ i^{'}=m+1,m+2,\cdots,m+j=r$.
For $j=m-1$, the resolution shows that  
   \[dd(\dot{P}(j+2),Q_r) =2= dd(\dot{P}(j+i),Q_r) \text{ for all } i=3,4,\cdots,m\] 
because  $I(j)=I(m-1)$ is not projective.                                                                 
Finally, for $j=m$,  $I(j-m+1)=I(1)$ is not projective and therefore
        \[dd(\dot{P}(j+2),Q_r) =3= dd(\dot{P}(j+i),Q_r) \text{ for all } i=3,4,\cdots,m.\] 
 Thus
 \[
    \begin{array}{ccl}
      dom.dimA & = & ddP(n-m+2)\\
                & = &
                       \begin{cases}
                       \frac{2n-2(m+j)}{m}+1 & j=1,2,\cdots,m-2 \\
                       \frac{2n-2(m+j)}{m}+2  & j=m-1\\
                      \frac{2n-2(m+j)}{m}+3  & j=m
\end{cases}\\
 &=&  \begin{cases}
       \frac{2n-(m+2j)}{m} & j=1,2,3,\cdots ,m-2\\
       \frac{2n-2j}{m} & j=m-1\\
       \frac{2n-j}{m} & j=m
    \end{cases}                     

    \end{array}
    \]
with $n\in m\N+j$.
\end{proof}

\section{Bound quiver algebras of general trees}
\label{sec: general tree}

Throughout the following we assume that  $Q\neq \A_n$ is a finite tree. 
\begin{ntn}
 We denote by $b(k,l)$ a branching vertex $b$ such that $b$ is the target of $k$ arrows and the source of $l$ arrows,
where $k,l\in\N$ are not necessarily equal.
Hence $b(k,0)$ and $b(0,l)$ denote receptively  the sink of $k$ arrows and the source of $l$ arrows.

\end{ntn}

\begin{lem}
Let $B$ be a bound quiver algebra of $Q$. If $Q$ contains a branching vertex $b(k,0)$ or $b(0,l)$ with $k,l\geq2$, then $dom.dimB=0$.
\end{lem}
\begin{proof}
 	If $Q$ contains a branching vertex $b(k,0)$ where $k\geq2$. Then $P(b)$ is
 simple and the $topI(b)$ contains at least $k\geq2$ simple summands. Thus $I(b)$ can not be  projective. Hence
$EP(b)=I(b)$ is not projective, and we have  $dom.dim B=0.$

Next we assume that $Q$ contains a branching vertex $b(0,l)$ where $l\geq2$. Then the socle of $P(b)$ is not simple; 
indeed $socP(b)=\bigoplus\limits_{1\leq i\leq l} S(a_i)$, where $a_i\in Q_0$ are such that there exists a maximal path
from $b$ to each $a_i$. Now $topI(a_i)$ contains the simple summand $ S(b)$, 
and thus $I(a_i)$ is not  projective. Hence
$EP(b)=\bigoplus\limits_{1\leq i\leq l}  I(a_i)$ is not projective, and again $dom.dim B=0.$
\end{proof}
\begin{dfn}
By a \emph{left arm} of a tree $Q$ we mean a subquiver of $Q$ which is linearly oriented $\A_n$ with $n\geq1$ from a source to an immediate predecessor
of a branching vertex. Similarly, a \emph{right arm} of $Q$ is a subquiver of $Q$ which is linearly oriented $\A_n$ with $n\geq1$ from an immediate successor
of a branching vertex to a sink. 
\end{dfn}

A left (right) arm is said to be trivial if it is $\A_n$ with $n=1$, otherwise it is called non-trivial. A tree is said to be a \emph{tree without arms}
if it has no non-trivial arms. 
We always denote by  $Q'=(Q_0',Q_1')$ a tree without arms, while the subset of $Q_0'$  obtained by dropping sources and sinks of $Q'$ will be 
denoted by $Q_0''$. Note that both $Q$ and $Q'$ always have the same number of sources and sinks.

\subsection{Trees without arms}
\label{sec:without arms}
 Throughout the section, $B'=KQ'/\mci'$ is a bound quiver algebra of a tree $Q'$ without arms, where 
$\mci'$ is an admissible ideal of  $KQ'$ generated by a set $R'$ of zero relations on $Q'$.

\begin{dfn}
\label{dfn:*}
Let  $Q'$ be a tree without arms.
The conditions $(\ast)$ on  $Q'$ are defined as

\begin{itr}
 \item For each source $a$ and sink $c$, both $P(a)$ and $I(c)$ are uniserial.
\item For each $i\in Q_0'$, $socP(i)=\oplus S(c)$, where each $c$ is a sink.
\item  For each $i\in Q_0'$, $topI(i)=\oplus S(a)$, where each $a$ is a source.
\end{itr}

\end{dfn}
\begin{rem}
 In fact, when $Q'$ has an equal number of sources and sinks, (ii)  implies (iii), and vice versa.
\end{rem}

 \begin{lem}
\label{lem:source-sink equal}
  If a set $R'$ of zero relations on $Q'$  satisfies the conditions $(\ast)$, then  $Q'$ has
 equal number of sources and sinks.
 \end{lem}
\begin{proof}
 Assume, to the contrary, that the number of sources is not equal to the number of sinks.
Let $X$ and $Y$ be the sets of sources and sinks of $Q'$, respectively.
 (i) and (ii) of the conditions $(\ast)$ imply that for each source $a$, $socP(a)=S(c)$
for some sink $c$. This defines a map, say,  $f:X\rightarrow Y$ sending each source $a$ to a
unique sink $c$.
 If the number of sources is greater than the number of sinks,
then there exist at least two sources, say $a_1$ and $a_2$, and a sink $c$ such that $f(a_1)=c=f(a_2)$.
This implies that $f$ is not injective. Consequently $I(c)$ is not uniserial, a contradiction to (i).

Dually, (i) and (iii) together define a map  $g:Y\rightarrow X$ by $g(c)=a$ such that 
 $topI(c)=S(a)$. Now if there are more sinks than sources, then there are at least two
sinks, say $c_1$ and $c_2$, and a source $a$ such that $g(c_1)=a=g(c_2)$. Hence 
$g$ is not injective as well, and consequently $P(a)$ is not uniserial, again
a contradiction to (i). Hence $Q'$ has  equal number of sources and sinks.

\end{proof}
An immediate consequence is the following 
\begin{cor}
\label{lem:conditions *}
 The conditions $(\ast)$ imply:
sources and sinks are in one-to-one correspondence:
 there is a unique maximal path from each source $a$ to a unique sink $c$ such that $P(a)\cong I(c)$.
\end{cor}
\begin{proof}
     Suppose conditions $(\ast)$ hold. It follows from Lemma \ref{lem:source-sink equal}, that $Q'$ has
equal number of sources and sinks. 
Let $a$ be a source in $Q'$. Then we have from  the conditions $(\ast)$ that 
$socP(a)=S(c)$ and  $topI(c)=S(a)$ where $c$ is a unique sink.
This implies that the path from $a$ to $c$ is maximal, and hence $P(a)\cong I(c)$.

\end{proof}

The reverse implication is not true in general, as shown in the following
\begin{exl}
Let $Q'$ be the following tree without arms:
 
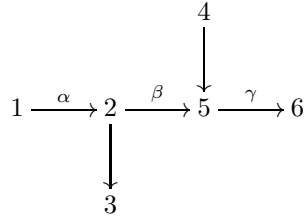
\[\begin{xy}
\xymatrix {
&&4\ar[d]\\
1\ar[r]^{\alpha}&2\ar[r]^{\beta}\ar[d]&5\ar[r]^{\gamma}&6\\
&3
}
\end{xy}\]
Let $\{\beta\alpha, \gamma\beta\}$ be a set of zero relations on  $Q'$. Then sources and sinks are in one-to-one
correspondence, but it does not imply the conditions $(\ast)$. For $socP(2)=S(3)\oplus S(5)$, while the vertex 5 is not a sink.
\end{exl}

 \begin{prp}
  If a set $R'$ of  relations on $Q'$ does not satisfy the conditions $(\ast)$, then $dom.dimB'=0$.
 \end{prp}
\begin{proof}
 We assume that $R'$ does not satisfy the conditions $(\ast)$.
First, let $a$ be a source such that $P(a)$ is not uniserial. Then $P(a)$ is not injective. Hence
 $EP(a)$ is not projective, because $a$ is a source. This gives  $dom.dimB'=0$.
Similarly, if $I(c)$ is not uniserial for some sink $c$, then obviously $I(c)$ is not projective
and so is  $EP(c)=I(c)$. Consequently, $dom.dimB'=0$.

Now we assume that $R'$ satisfies (i) but does not satisfy (ii). Then obviously (iii) is also not satisfied.
Because (ii) does not hold, there exist two leftmost vertices, say,  $i$ and $h$ such that $socP(i)\supseteq S(h)$, where
$h$ is not a sink.
Now if there exist a source $j$ and  a path $x$ from  $j$ to $h$, then $x$ is not zero in the algebra $B'$ because 
otherwise it would contradict the fact that $h$ is leftmost. This shows that  $topI(h)$ containing $S(j)$
and $S(i)$ is not simple. Hence $EP(i)\supseteq I(h)$ is not projective, and $dom.dimB'=0$.

If no such $j$ and $x$ exist, then  $topI(h)=S(i)$ is simple. Since $Q'$ has no arms, either $i$ or $i^{+}$ is a branching 
vertex $b=b(1,l\geq2)$. If $i=b$, then $I(h)$ is not projective, since $topI(h)=S(i)$ but $P(i)$ is not uniserial.
Hence $EP(i)\supseteq I(h)$ is not projective. If $i^{+}=b$, then $i$ must be a source of 
$Q'$, since $i$ is leftmost. Thus $P(i)$ is uniserial by (i). Now there exists at least one successor $h'$ of 
$i^{+}=b$ such that $topI(h')=S(i^{+})$, whereas $P(i^{+})$ is not uniserial.
Hence $EP(i^{+})\supseteq I(h')$ is not projective, showing that  $dom.dimB'=0$.
\end{proof}

\begin{lem}
\label{lem:injective envelope}
If a set $R'$ of zero relations on  $Q'$ satisfies the conditions $(\ast)$, then $EP(a)$  is projective
for each  $a\in Q_0'$.
\end{lem}
\begin{proof}
 Let $a\in Q_0'$ be an arbitrary vertex.  
Then by (ii) of  the conditions $(\ast)$,
$socP(a)=\oplus S(j)$ where each $j$ is a sink. This gives
$EP(a)=\oplus I(j)$. Because each $j$ is a sink and $R'$ satisfies the conditions
$(\ast)$, it follows from Corollary \ref{lem:conditions *} that each $I(j)$ is 
projective. Hence $EP(a)$ is projective for every $a\in Q_0'$.
\end{proof}

\begin{prp}
\label{prp:dd1}
If a set $R'$ of zero relations on $Q'$ satisfies the  conditions $(\ast)$, then $dom.dimB' =1$.
\end{prp}

\begin{proof}
 By Lemma \ref{lem:source-sink equal}, $Q'$ has equal number of sources and sinks. Since $R'$ satisfies the conditions $(\ast)$, it follows immediately from
Lemma \ref{lem:injective envelope} that the injective envelope $EP(a)$ of $P(a)$ is projective
 for each vertex $a \in Q_0'$. This implies that $dom.dimB'\geq1$.

Next we show that $dom.dimB'=1$. Let $c\in Q_0'$ be a sink. Then there exists a unique maximal path from the corresponding source $a$ to $c$ such that
\[P(a)\cong I(c)=\begin{matrix}
a\\a^{+}\\\vdots\\c^{-}\\c
\end{matrix}\]
where $ a^{+}$ is the immediate successor of $a$ and $c^{-}$ is the immediate predecessor of $c$. In fact $c^{-}$
 is a branching vertex $b$, since $Q'$ has no arms. The minimal injective resolution of $P(c)$ becomes
	\[
0\rightarrow
\begin{matrix}
c
\end{matrix}\rightarrow
\begin{matrix}
a\\
a^{+}\\
\vdots\\
b\\
c
\end{matrix} \rightarrow
\begin{matrix}\\
a\\
a^{+}\\
\vdots\\
b
\end{matrix}
\rightarrow 0  \]
or \[0\rightarrow P(c)\rightarrow I(c)\rightarrow I(b)\rightarrow 0\]
where $ I(b)$ is not projective. This proves that $dom.dimB'=1$.
\end{proof}

We summarize this section as 
\begin{thm}
\label{thm:without arms}
Let $R'$ be a set of zero relations on $Q'$. Then
    
      \[
       dom.dimB' =
    \begin{cases}
      1 &\text{ if } R'\text{ satisfies the conditions } (\ast)\\
       0 &\text{ otherwise.}
       \end{cases}
    \]
\end{thm}

\subsection{Trees with arms}
\label{sec:with arms}
In this subsection, a bound quiver algebra $KQ/\mci$ of a tree $Q$ with arms  will be denoted simply by $B$, where
$\mci$ is an admissible ideal of  $KQ$ generated by a set $R$ of zero relations on $Q$.

In general,  $dom.dimB $ is not equal to $dom.dimB'$, but we have the following

\begin{prp}
\label{prp:ddB=0} 
Let $R$ and $R'$ be sets of zero relations on $Q$ and $Q'$  respectively.

\begin{itemize}
 \item[(a)] If $R=R'$, then $dom.dimB = dom.dimB'$.
\item[(b)]  If $R'\subsetneq R$ and $R\cap S'=R'$, then $dom.dimB' = 0$ implies $dom.dimB= 0$,
where  $S'$ is the set of all possible zero relations on $Q'$.
\end{itemize}
\end{prp}

\begin{proof}
 (a)  First we assume that  $dom.dimB' = 0$. Then clearly $R'$ does not satisfy the conditions $(\ast)$.
Hence there exists at least one $i\in Q_0'$ such that the injective envelope $EP(i)$ of the
$B'$-module $P'(i)$ is not projective. Since $R=R'$ and $Q_0'\subseteq Q_0$, it follows that the injective envelope $EP(i)$ 
of the $B$-module $P(i)$ is not projective where $i\in Q_0$. Hence  $dom.dimB = 0 $.

Now assume that $dom.dimB' = 1$. Then by Theorem \ref{thm:without arms}, $R'$ satisfies the conditions $(\ast)$.
 Since $R=R'$, so $R$ also satisfies the condition $(\ast)$. It follows from the Lemma \ref{lem:injective envelope} that 
the injective envelope $EP(i)$ of a $B'$-module $P(i)$ is projective for each $i\in Q_0'$. This implies that 
the injective envelope $EP(i)$ of a $B$-module $P(i)$ is projective for each $i\in Q_0$, because $R=R'$ and $Q_0'\subseteq Q_0$. 
Consequently, $dom.dimB = 1$. 

(b)  We suppose that $dom.dim B' = 0$. 
Then it follows from Theorem \ref{thm:without arms} that  $R'$ does not satisfy the conditions $(\ast)$.
Hence there exists at least one $i\in Q_0'$ such that the injective envelope $EP(i)$ of the $B'$-module $P(i)$ is not projective.
Now $R'\subsetneq R$ and $R\cap S'=R'$ imply that  $R$ does not satisfy the conditions $(\ast)$ on the subquiver $Q'$ of $Q$.
This implies that the injective envelope $EP(i)$ of the $B$-module $P(i)$ is not projective, where $i\in Q_0$ because $Q_0'\subseteq Q_0$,
 and hence $dom.dim B = 0$. 
\end{proof}

 The reverse implication in  part (b) of Proposition \ref{prp:ddB=0} is not true in general. For we have the following
\begin{exl}
\label{exl:counter1}
 Let $Q$ be the following tree with arms:
 
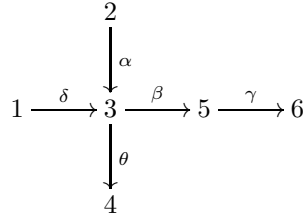
\[\begin{xy}
\xymatrix {
&2\ar[d]^{\alpha}\\
1\ar[r]^{\delta}&3\ar[r]^{\beta}\ar[d]^{\theta}&5\ar[r]^{\gamma}&6\\
&4
}
\end{xy}\]
Let $R'=\{\theta\alpha,\beta\delta\}$ and $R =\{\theta\alpha,\beta\delta,\gamma\beta\alpha\}$
be two sets of zero relations. Then $dom.dim B = 0$, since $EP(3)$ contains a non-projective summand $I(6)$,
but $R'$ satisfies the conditions $(\ast)$ on $Q'$, and thus $dom.dimB' = 1$.

This example also shows that if $dom.dim B' = 1$ and $R'\subsetneq R$, then $dom.dim B$ is not necessarily equal to one.
\end{exl}

\begin{lem}
\label{lem:uniserial}
Let $R$ and $R'$ be sets of zero relations on $Q$ and $Q'$ respectively, 
such that $R'$ satisfies the conditions $(\ast)$. If $R'\subseteq R$, then
  $B$-modules $P(i)$ and $I(i)$ are uniserial for each $i\in Q_0\smallsetminus Q_0''$.
\end{lem}
\begin{proof}
  We assume that $R'\subseteq R$. Let $i$ be arbitrary in $Q_0\smallsetminus Q_0''$. There are two cases: 
either $i$ is contained in a left arm or it belongs to some right arm.
  First, we suppose that $i$ is contained in a left arm of $Q$. Then it is trivial to see that $I(i)$ is uniserial,
 because arms of $Q$, by definition, are linearly oriented. Also $socP(i)$ is necessarily simple $S(j)$ 
  for some successor $j\in Q_0$ of $i$, because $R'$ satisfies the conditions $(\ast)$ and $R'\subsetneq R$. 
Hence $P(i)$ is uniserial. 

Next, suppose that $i$ belongs to some right arm of $Q$. Then trivially 
  $P(i)$ is uniserial, because arms are linearly oriented. Again by the same argument that $R'$ satisfies the
 conditions $(\ast)$ and $R'\subsetneq R$, it follows that $topI(i)$ is simple $S(h)$ for some predecessor
  $h\in Q_0$ of $i$, and thus $I(i)$ is uniserial. Hence both projective and 
 injective $B$-modules $P(i)$ and $I(i)$ are uniserial for each $i\in Q_0\smallsetminus Q_0''$.
\end{proof}

 We note that the assumptions in Lemma \ref{lem:uniserial} are not sufficient for the indecomposable projective 
(injective) $B$-module $P(a)$  ($I(c)$)  to be injective (projective) for each source
  $a$ (sink $c$) in $Q$, as shown in the following
 
\begin{exl}
\label{exl:counter2}
Let $Q$ be the following tree with arms:
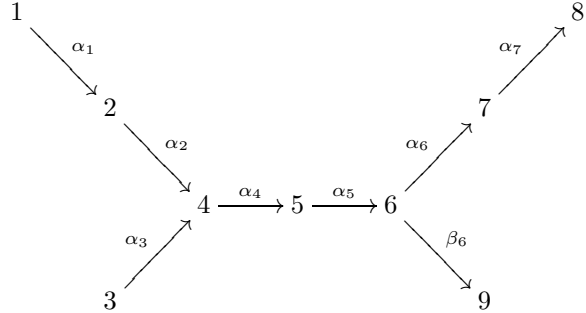
\[\begin{xy}
\xymatrix {
1\ar[rd]^{\alpha_1}&&&&&&8\\
&2\ar[rd]^{\alpha_2}&&&&7\ar[ru]^{\alpha_7}\\
&&4\ar[r]^{\alpha_4}&5\ar[r]^{\alpha_5}&6\ar[ru]^{\alpha_6} \ar[rd]^{\beta_6}\\
&3\ar[ru]^{\alpha_3}&&&&9
}
\end{xy}\]
Let 
\[
\begin{array}{ccl}
R'& =& \{\beta_6\alpha_5\alpha_4\alpha_2,\alpha_6\alpha_5\alpha_4\alpha_3\}\\
R    & = & \{\beta_6\alpha_5\alpha_4\alpha_2,\alpha_6\alpha_5\alpha_4\alpha_3,\alpha_5\alpha_4\alpha_2\alpha_1,\alpha_7\alpha_6\alpha_5\alpha_4\alpha_2\}
\end{array}
\]
be two sets of zero relations. Then clearly $R'$ satisfies the conditions $(\ast)$ 
and $R'\subset R$. We see that the $B$-modules
 $P(i)$ and $I(i)$ are uniserial for each $i\in Q_0\smallsetminus Q_0''=\{1,2,3,7,8,9\}$.
But, in particular, for source 1 and sink 8,
\[P(1)=\begin{matrix}
1\\
2\\
4\\
5
\end{matrix}
\quad
\text{ and }
\quad
I(8)=\begin{matrix}
4\\
5\\
6\\
7\\
8
\end{matrix}\]
are  not projective-injective, where $4,5\in Q_0''$. We also note that the path from 1 to 5 is maximal, but $P(1)\ncong I(5)$, because 
\[I(5)=\begin{matrix}
\quad\quad\quad1\\
3\quad\quad2\\
\diagdown\diagup\\
4\\
5
\end{matrix}\]
is not uniserial.
\end{exl}

We tackle this problem by defining on trees with arms the following natural analogue of the conditions $(\ast)$.
\begin{dfn}
\label{dfn:**}
 Let $R$ and $R'$ be sets of zero relations on $Q$ and $Q'$, respectively, 
such that $R'\subseteq R$. Then $R$ is said to satisfy the conditions $(\ast\ast)$  if
 \begin{itr}
\item $R'$ satisfies the conditions $(\ast)$.
\item $\forall $ $i$ in left arm, $socP(i)=S(i')$ for some successor $i'\notin Q_0''$ of $i$.
\item  $\forall $ $j$ in right arm, $topI(j)=S(j')$ for some predecessor $j'\notin Q_0''$ of $j$.
\end{itr}
\end{dfn}

\begin{rem}
\label{rem:**}
 It is important to mention that when $R$ satisfies (i) the conditions $(\ast\ast)$, then 
(ii) implies  (iii), and vice versa. For if $i$ in a left arm is such that $socP(i)=S(j)$
where $j$ belongs to some right arm. Then $topI(j)$ is necessarily simple $S(h)$ with
$h$ either $i$ or some of its predecessors, since otherwise it would contradict the
conditions $(\ast)$. Similarly, (iii) implies  (ii) can be justified.
\end{rem}

 \begin{lem}
\label{lem:equal ss in Q}
    If a set $R$ of zero relations on $Q$  satisfies (i) of the conditions $(\ast\ast)$, 
then $Q$ has equal number of sources and sinks.
\end{lem}
\begin{proof}
 If $R$ satisfies (i) of the conditions $(\ast\ast)$, 
then it follows from Lemma \ref{lem:source-sink equal} that $Q'$ has equal number of sources and sinks.
Since both $Q$ and $Q'$ always have the same number of sources and sinks, so the Lemma follows.
\end{proof}

\begin{prp}
 Let $R$ and $R'$ be sets of zero relations on $Q$ and $Q'$ respectively, 
such that $R'\subseteq R$ and $R\cap S'=R'$. If $R$ does not satisfy the conditions $(\ast\ast)$,
then $dom.dimB =0$.
\end{prp}
\begin{proof}
We assume that $R$ does not satisfy the conditions $(\ast\ast)$.
Suppose $R$ does not satisfy (i) of the conditions $(\ast\ast)$, that is, $R'$ does not satisfy the condition $(\ast)$. 
Then it follows from Theorem \ref{thm:without arms}  that $dom.dimB'=0$.
Since $R\cap S'=R'$, it follows immediately from Proposition \ref{prp:ddB=0} that $dom.dimB=0$.

Now we assume that (i) holds but $R$ does not satisfy (ii) of the conditions $(\ast\ast)$. 
Then, by Remark \ref{rem:**}, (iii) is also not satisfied by $R$. Since (i) holds, it follows from Lemma \ref{lem:equal ss in Q}
that  $Q$ has equal number of sources and sinks.
Let an $i$ in a left arm be such that $socP(i)=S(j)$ for some $j\in Q_0''$. Because  $j\in Q_0''$ and (i) holds, so $I(j)$ 
is not uniserial, and thus it can not be projective. Hence $EP(i)=I(j)$ is not projective, and ultimately we get $dom.dimB=0$.
\end{proof}

\begin{lem}
\label{lem:proj-inj}
Let $R$ and $R'$ be sets of zero relations on $Q$ and $Q'$ respectively, 
such that $R'\subseteq R$. If $R$ satisfies the conditions $(\ast\ast)$,
 then for each source $a$ and sink $c$ in $Q$, $B$-modules $P(a)$ and $I(c)$ are projective-injective.
\end{lem}
\begin{proof}
 We assume that $R$ satisfies the conditions $(\ast\ast)$. Let $a$ be an arbitrary source in $Q$. 
Then from (ii)   of the conditions $(\ast\ast)$,  $socP(a)=S(j)$ for some successor $j\in Q_0\smallsetminus Q_0''$ of $a$.
 Because $j\notin Q_0''$, $I(j)$ is uniserial by Lemma \ref{lem:uniserial}.  Now the path starting from the source
  $a$ to $j$ is necessarily maximal, and hence $P(a)\cong I(j)$ by Lemma \ref{lem:new maximal path}.

  Next we show that $I(c)$ is projective for each sink $c$ in $Q$. Obviously, $c$ belongs to some right arm. 
 It follows from (iii) of the conditions $(\ast\ast)$
  that $topI(c)=S(i)$ for some predecessor $i\in Q_0\smallsetminus Q_0''$ of $c$. By Lemma \ref{lem:uniserial}, $P(i)$ is uniserial.
 Since the path from  $i$ to the sink $c$ is maximal, it follows  that $P(i)\cong I(c)$. 
 \end{proof}

\begin{lem}
 \label{lem:injective envelope arms}
 Let $R$ and $R'$ be sets of zero relations on $Q$ and $Q'$ respectively, 
such that $R'\subseteq R$. If $R$  satisfies the conditions $(\ast\ast)$, 
then  $EP(i)$ is projective for each $i\in Q_0$.
\end{lem}
\begin{proof}
  We assume that $R$ satisfies the conditions $(\ast\ast)$. Let $i\in Q_0$ be an arbitrary 
vertex.

First, let $i$ belong to some left arm of $Q$. 
From (ii) of the conditions $(\ast\ast)$, $socP(i)=S(j)$ for some successor 
$j\in Q_0\smallsetminus Q_0''$. It follows from Lemma \ref{lem:uniserial}, that $I(j)$ is also uniserial. 
  Now if the path from $i$ to $j$ is maximal, then Lemma \ref{lem:new maximal path} gives  $P(i)\cong I(j)$.
 Otherwise,  there exists in the left arm a predecessor $h$ of $i$ such that 
  the path from $h$ to $j$ is maximal. It follows again from 
  Lemma \ref{lem:new maximal path}, that $P(h)\cong I(j)$. This implies that $EP(i)=I(j)$ is projective. 
Hence $EP(i)$ is projective for each $i$ in a left arm.

Now suppose that  $i$ does not belong to any left arm of $Q$.
 Then $P(i)$ may or may not be uniserial. As the conditions $(\ast)$ also hold, we can assume that $socP(i)=\oplus S(j)$, 
 where each $j$ belongs to some right arm  and is the target of some maximal path. 
 From (iii) of the conditions $(\ast\ast)$, 
we have  $topI(j)=S(h)$ for some predecessor
 $h\notin Q_0''$ of  $j$. Since each $P(h)$ is uniserial and 
paths from $h$ to $j$ are maximal, we have  $P(h)\cong I(j)$ by Lemma \ref{lem:new maximal path}.
 Hence $EP(i)=\oplus I(j)$ is projective. 
\end{proof}

\begin{prp}
 Let $R$ and $R'$ be sets of zero relations on $Q$ and $Q'$ respectively, 
such that $R'\subseteq R$. If $R$ satisfies the conditions $(\ast\ast)$,
then $dom.dimB =1$.
\end{prp}
\begin{proof}
 First we assume that $R$ satisfies the conditions $(\ast\ast)$. Then it follows immediately from Lemma 
\ref{lem:injective envelope arms} that the injective envelope $EP(a)$ 
of a $B$-module $P(a)$ is projective for each vertex $a$ in $Q_0$. This implies that $dom.dimB\geq1$.

  Now to show that $dom.dimB\ngtr1$, let $b\in Q_0''$ be a branching vertex such that $P(b)$ is not uniserial.
 $I(b)$ is also not uniserial, since $b\in Q_0''$ and $R'\subsetneq R$. We assume that $socP(b)=S(j_1)\oplus S(j_2)$, 
where $j_1\neq j_2$, belonging to two distinct right arms, are the targets of some maximal paths. Therefore, 
there exist $i_1\neq i_2$ in the respective two distinct left arms such that
\[P(i_1)\cong I(j_1)=\begin{matrix}
i_1\\
i_1^{+}  \\
\vdots\\
b\\
\vdots\\
j_1
\end{matrix}
\quad
\text{ and }
\quad
P(i_2)\cong I(j_2)=\begin{matrix}
i_2\\
i_2^{+}  \\
\vdots\\
b\\
\vdots\\
j_2
\end{matrix}\]
  where $i_1^{+}$ and $i_2^{+}$ are the immediate successors of $i_1$ and $i_2$ respectively. 
  We obtain the minimal injective resolution of $P(b)$ as \[0\rightarrow P(b)\rightarrow I(j_1)\oplus I(j_2)\rightarrow I(b)\rightarrow \]
  where $I(b)$ is not projective, and hence $dom.dimB=1$.

\end{proof}

We summarize the case of trees with arms as 
\begin{thm}
\label{thm:with arms}
 Let $R$ and $R'$ be sets of zero relations on $Q$ and $Q'$ respectively, 
such that $R'\subseteq R$ and $R\cap S'=R'$. Then
 
     \[
       dom.dimB =
    \begin{cases}
       1 &\text{ if } R \text{ satisfies the conditions } (\ast\ast)\\
	0 &\text{ otherwise. }
       \end{cases}
    \]
  
\end{thm}

A straightforward consequence is the following
\begin{cor}
Let $R$ and $R'$ be sets of zero relations on $Q$ and $Q'$ respectively, 
such that $R'\subseteq R$ and $R\cap S'=R'$. Then{\small
     \[
       dom.dimB =
    \begin{cases}
       dom.dimB'=1 &\text{ if } R \text{ satisfies the conditions } (\ast\ast)\\
	dom.dimB'=0 &\text{ if } R' \text{ does not satisfy the conditions } (\ast).
       \end{cases}
    \]}

\end{cor}

\begin{ack-plain}
 The author is heartily thankful to his PhD advisor Prof. Dr. Steffen Koenig (at IAZ, Universit\"at Stuttgat) for the valued suggestions and comments
that enabled the author to improve all the early drafts of this article.
\end{ack-plain}

\bigskip

Present address:\\
Institut f\"ur Algebra und Zahlentheorie (IAZ),\\ Universit\"at Stuttgart, Pfaffenwaldring 57, D-70569 Stuttgart, Germany.\\
e-mail: \texttt{mabrar@mathematik.uni-stuttgart.de}

Permanent address:\\
Department of Mathematics,\\ Kohat University of Science and Technology (KUST). KPK, 26000 Kohat, Pakistan.\\
e-mail: \texttt{abrarrao@gmail.com}

\end{document}